\documentclass[a4paper,11pt]{article}
\usepackage{latexsym}
\usepackage{amssymb}
\usepackage{amsfonts}
\usepackage{amsmath}
\usepackage[dvips]{graphicx}
\usepackage{epsf}
\newcommand{\qed}{$\Box$}

\newenvironment{@abssec}[1]{%
    \if@twocolumn

      \section*{#1}%
    \else

      \vspace{.05in}\footnotesize
      \parindent .2in
 {\upshape\bfseries #1. }\ignorespaces
    \fi}

    {\if@twocolumn\else\par\vspace{.1in}\fi}

\newenvironment{keywords}{\begin{@abssec}{\keywordsname}}{\end{@abssec}}

\newenvironment{AMS}{\begin{@abssec}{\AMSname}}{\end{@abssec}}

\newcommand\keywordsname{Key words}
\newcommand\AMSname{AMS subject classifications}
\newcommand\AMname{AMS subject classification}
\newtheorem{theorem}{Theorem}
 \newtheorem{lemma}[theorem]{Lemma}

\newtheorem{remark}[theorem]{Remark}
 
\def\qed{\vbox{\hrule height0.6pt\hbox{%
  \vrule height1.3ex width0.6pt\hskip0.8ex
  \vrule width0.6pt}\hrule height0.6pt
 }}
\def\theequation{\arabic{section}.\arabic{equation}}
 \def\thetheorem{\arabic{section}.\arabic{theorem}}

\def\theequation{\arabic{section}.\arabic{equation}}
 \def\thetheorem{\arabic{section}.\arabic{theorem}}

\def\veps{\varepsilon}

\def\RE{\mathbb R}

\def\ovr{\overline}

\def\pa{\partial}

\def\ga{\gamma}
\def\Om{\Omega}

\def\supp{\mbox{\rm supp }}

\title{Interaction between nonlinear diffusion\\
and geometry of domain
\thanks{This research was partially supported by a Grant-in-Aid
for Scientific Research (B) ($\sharp$ 20340031) of
Japan Society for the Promotion of Science and by a
Grant of the Ital\-ian MURST.}}

\author{Rolando Magnanini\thanks{Dipartimento di Matematica U.~Dini,
Universit\` a di Firenze, viale Morgagni 67/A, 50134 Firenze, Italy.
({\tt magnanin@math.unifi.it}).} 
and Shigeru Sakaguchi\thanks{Department of Applied Mathematics,
Graduate School of  Engineering, Hiroshima
University, Higashi-Hiroshima, 739-8527,  Japan.
({\tt sakaguch@amath.hiroshima-u.ac.jp}).}}

\begin{document}

\maketitle

\begin{abstract}
 Let  $\Omega$ be a domain in $\mathbb R^N$, where $N \ge 2$ and $\partial\Omega$ is not necessarily bounded. We consider  nonlinear diffusion equations of the form  $\partial_t u= \Delta \phi(u)$. Let $u=u(x,t)$ be the solution of either the initial-boundary value problem over $\Omega$, where the initial value equals zero and the boundary value equals $1$,  or 
the Cauchy problem where the initial data is  the characteristic function of the set $\mathbb R^N\setminus \Omega$. 

We consider  an open ball $B$  in $\Omega$ whose closure intersects $\partial\Omega$ only at one point, and
we derive asymptotic estimates for the content of substance 
in $B$ for short times in terms of  geometry of $\Omega$.
Also, we obtain a characterization of the hyperplane involving a stationary level surface of  $u$ by using the sliding method due to Berestycki, Caffarelli, and Nirenberg. These results tell us about interactions between nonlinear diffusion and geometry of domain.
\end{abstract}


\begin{keywords}
nonlinear diffusion, geometry of domain, initial-boundary value problem, Cauchy problem, initial behavior.
\end{keywords}

\begin{AMS}
Primary 35K55, 35K60; Secondary   35B40.
\end{AMS}

\pagestyle{plain}
\thispagestyle{plain}
\markboth{R. MAGNANINI AND S. SAKAGUCHI}{Nonlinear diffusion
and geometry of domain}

\pagestyle{plain}
\thispagestyle{plain}

\section{Introduction}
 Let  $\Omega$ be a $C^2$ domain in $\mathbb R^N,$ where $N \ge 2$ and $\partial\Omega$ is not necessarily bounded, and let $\phi : \mathbb R \to \mathbb R$ satisfy
\begin{equation}
\label{nonlinearity}
\phi \in C^2(\mathbb R), \quad \phi(0) = 0, \ \mbox{ and }\ 0 < \delta_1 \le \phi^\prime(s) \le \delta_2\ \mbox{ for } s \in \mathbb R,
\end{equation}
 where $\delta_1, \delta_2$ are positive constants. Consider the unique bounded solution $u = u(x,t)$ of either
  the  initial-boundary value problem: 
\begin{eqnarray}
&\partial_t u=\Delta \phi(u)\ \ &\mbox{in }\ \Omega\times (0,+\infty),\label{diffusion}\\
&u=1\ \ &\mbox{on }\ \partial\Omega\times (0,+\infty),\label{dirichlet}\\
&u=0\ \ &\mbox{on }\ \Omega\times \{0\},\label{initial}
\end{eqnarray}
or the Cauchy  problem:
\begin{equation}
\label{cauchy}
\partial_t u=\Delta \phi(u)\ \mbox{ in }\ \mathbb R^N \times (0, +\infty)\quad\mbox{ and }\ u = \chi_{\Omega^c}\ \mbox{ on }\ \mathbb R^N \times \{0\};
\end{equation}
here $\chi_{\Omega^c}$ denotes the characteristic function of the set $\Omega^c = \mathbb R^N \setminus \Omega$. Note that
the uniqueness of the solution of either problem \eqref{diffusion}-\eqref{initial} or \eqref{cauchy} follows from the 
comparison principle (see Theorem \ref{th:comparison principle} in the present paper).
Since $\partial\Omega$ is of class $C^2$, we can construct barriers at any point on the boundary $\partial\Omega\times (0,+\infty)$ for problem \eqref{diffusion}-\eqref{initial}.  Thus, by the theory of uniformly parabolic equations (see \cite{LSU}), we have  the existence of a solution $u \in C^{2,1}(\Omega\times(0,+\infty)) \cap L^\infty(\Omega \times (0,+\infty))\cap C^0(\overline{\Omega} \times (0,+\infty)) $ such that $u(\cdot,t) \to 0$ in $L^1_{loc}(\Omega)$ as $t \to 0$ for problem \eqref{diffusion}-\eqref{initial}.  For problem  \eqref{cauchy},  since for any bounded measurable initial data there exists a  bounded solution of the Cauchy problem for $\partial_t u =\Delta \phi(u)$ by the theory of uniformly parabolic equations, we always have a solution $u \in C^{2,1}(\mathbb R^N \times (0,+\infty))\cap L^\infty(\mathbb R^N \times (0,+\infty))$ such that $u(\cdot,t) \to \chi_{\Omega^c}(\cdot)$ in $L^1_{loc}(\mathbb R^N)$ as $t 
\to 0$ for any domain  $\Omega$, that is, in the case of problem \eqref{cauchy}, we only need that the set $\Omega$ is measurable.

\par
The differential equation in \eqref{diffusion} or in \eqref{cauchy} has
the property of {\it infinite} speed of propagation of disturbances from rest, since  
\begin{equation}
\label{infinite speed}
\int_0^1 \frac {\phi^\prime(\xi)}\xi d\xi = + \infty, 
\end{equation}
as it follows from \eqref{nonlinearity}.
\par
By  the strong comparison principle, we know that
\begin{equation*}
\label{bounds}
0 < u < 1\ \mbox{either  in } \Omega \times (0, +\infty) \mbox{ or in } \mathbb R^N \times(0,+\infty);
\end{equation*}
also, as $t\to 0^+,$ $u$ exhibits a {\it boundary layer:} while $u\to 0$ in $\Om,$ $u$ remains
equal to $1$ on $\pa\Om.$ 
The profile of $u$ as $t\to 0^+$ is controlled by the function $\Phi$ defined by
\begin{equation}
\label{definition of Phi}
\Phi(s) = \int_1^s \frac {\phi^\prime(\xi)}{\xi} d\xi\quad\mbox{ for } s > 0.
\end{equation}
In fact, in \cite[Theorem 1.1 and Theorem 4.1]{MSpoincare} we showed that, 
if $\pa\Om$ is bounded and $u$ is the solution of either problem {\rm (\ref{diffusion})-(\ref{initial})}  or problem {\rm (\ref{cauchy})}, then
\begin{equation}
\label{varadhan formula}
\lim_{t \to 0^+} -4t\Phi(u(x,t)) = d(x)^2 \ \mbox{ uniformly on every compact subset of } \Om.
\end{equation}
Here, $d = d(x)$ is the distance function:
\begin{equation}
\label{definition of distance}
d(x) = \mbox{ dist}(x, \partial\Omega)\quad\mbox{ for }  x \in \Omega.
\end{equation}
\par
Formula \eqref{varadhan formula} generalizes one obtained by Varadhan \cite{Va}
for the heat equation (and quite general linear parabolic equations); in
that case, $\Phi(s) = \log s$ since $\phi(s)\equiv s;$ 
\eqref{varadhan formula} tells us about an interaction between nonlinear diffusion and geometry of domain, since the function $d(x)$ is deeply related to  geometry of $\Omega$.
\par
We point out that \eqref{varadhan formula} was proved in \cite{MSpoincare} when
$\pa\Om$ is bounded. In Theorem \ref{th:nonlinear varadhan} in Section \ref{section simple application}, we will show how to extend its validity
to the case in which $\pa\Om$ is unbounded. Moreover,
with Theorem \ref{th:nonlinear varadhan} in hand, in Theorem \ref{th:hyperplane characterization} we obtain a characterization of 
hyperplanes as stationary level surfaces of the solution $u$
(i.e. surfaces where $u$ remains constant at any given time); 
this result generalizes one of those obtained in \cite{MSindiana, MSjde} for the heat equation.
As in \cite[Theorem 3.4]{MSindiana}, the proof still relies on the sliding method 
due to Berestycki, Caffarelli, and Nirenberg \cite{BCN} but,
by a different argument, allows us to treat more general assumptions on $\Om.$ 

Let us now state our main theorem which shows 
a more intimate link between short-time nonlinear diffusion and the geometry of the domain $\Om.$


\begin{theorem}
\label{th:interaction curvatures} Let $u$ be the solution of 
either problem {\rm (\ref{diffusion})-(\ref{initial})} or problem {\rm (\ref{cauchy})}.
Let $x_0 \in \Omega$ and assume that the open ball $B_R(x_0)$ centered at $x_0$ and with radius $R$ 
is contained in  $\Omega$ and such that $\overline{B_R(x_0)} \cap \partial\Omega = \{ y_0 \}$ for some $y_0 \in \partial\Omega$.
\par
Then we have:
\begin{equation}
\label{asymptotics and curvatures}
\lim_{t\to 0^+}t^{-\frac{N+1}4 }\!\!\!\int\limits_{B_R(x_0)}\! u(x,t)\ dx=
c(\phi,N)\left\{\prod\limits_{j=1}^{N-1}\left[\frac 1R - \kappa_j(y_0)\right]\right\}^{-\frac 12}.
\end{equation}
Here, 
$\kappa_1(y_0),\dots,\kappa_{N-1}(y_0)$ denote the principal curvatures of $\partial\Omega$ at $y_0$ with 
respect to the inward normal direction to $\partial\Omega$  
and $c(\phi,N)$ is a positive constant depending only on $\phi$ and $N$ 
(of course,  $c(\phi,N)$ depends on the problems  {\rm (\ref{diffusion})-(\ref{initial})}  or {\rm (\ref{cauchy})}). 
\par
When $\kappa_j(y_0) = \frac 1R$ for some $j \in \{ 1, \cdots, N-1\}$, 
the formula \eqref{asymptotics and curvatures} holds by setting the right-hand side to $+\infty$ (notice that 
$\kappa_j(y_0) \le 1/R$
for every $j \in \{ 1, \cdots, N-1\}$).
\end{theorem}

\begin{remark}{\rm
\label{concern regularity Th1.1} In view of the proof given in the end of Section \ref{section3},
under the existence of the solution $u$ of problem \eqref{diffusion}-\eqref{initial}, 
we need not assume that the entire $\partial\Omega$ is of class $C^2$ but only that 
it is of class $C^2$ in a neighborhood of the point $y_0$. Of course,
in the case of problem \eqref{cauchy} we only need to assume that $\partial\Omega$ is of class $C^2$ in a neighborhood of $y_0$.}
\end{remark}
\par
A version of Theorem \ref{th:interaction curvatures} was proved in \cite{MSprsea} for problem \eqref{diffusion}-\eqref{initial}, under the assumptions that $\partial\Omega$ is bounded and $\phi$ satisfies either 
$\int_0^1 \frac {\phi^\prime(\xi)}\xi d\xi < +\infty$ or $\phi(s) \equiv s.$
The reason why we could not treat cases in which $\int_0^1 \frac {\phi^\prime(\xi)}\xi d\xi = +\infty$ and $\phi$ is nonlinear
was merely technical. 
To be precise, in \cite{MSprsea}, the construction of supersolutions and subsolutions to problem \eqref{diffusion}-\eqref{initial}
was eased by the property of {\it finite} speed of propagation of disturbances from rest 
that descends from the assumption $\int_0^1 \frac {\phi^\prime(\xi)}\xi d\xi < +\infty.$ 
In fact, such barriers were constructed in a set $\Omega_{\rho} \times (0,\tau],$  with
\begin{equation}
\label{neighborhood of boundary}
\Omega_{\rho} =\{ x \in \Omega\ :\ d(x) < \rho \},
\end{equation}
where $\rho$ and $\tau$ were chosen sufficiently small so that 
the solution $u$ equals zero on the set $\Gamma_\rho\times (0,\tau],$ with 
\begin{equation}
\label{level surfaces of distance functions}
\Gamma_{\rho} =\{ x \in \Omega\ :\ d(x) =\rho \}.
\end{equation}

This property does not occur when \eqref{infinite speed} is in force. However, formula \eqref{asymptotics and curvatures} seems general and is expected to hold for general diffusion equations. Here, we in fact overcome 
some of those technical difficulties and prove \eqref{asymptotics and curvatures} for a class of nonlinear diffusion equations satisfying \eqref{infinite speed}; moreover,
the method of the proof of the present article enables us to treat also the case in which $\partial\Omega$ is unbounded.  
To be more specific, we construct the supersolutions and subsolutions for $u$ without using
the linearity of the heat equation and the result of Varadhan \cite{Va} 
as done in \cite{MSprsea}, but instead we exploit Theorem  \ref{th:nonlinear varadhan} together with a result of 
Atkinson and Peletier \cite[Lemma 4, p. 383]{AtP} concerning the asymptotic behavior of 
one-dimensional similarity solutions (see \eqref{limiting behavior} in the present paper). 
Then, as in \cite{MSprsea}, we take advantage of their explicit form $f_\pm(t^{-\frac 12} d(x))$  (see Lemmas \ref{barriers for IBVP} and \ref{barriers for Cauchy} in the present paper) to calculate their integrals over the ball $B_R(x_0)$ with the aid of the co-area formula. 
The proof of Theorem \ref{th:interaction curvatures} is finally completed by letting $t \to 0^+$ 
and using a geometric lemma \cite[Lemma 2.1, p. 376]{MSprsea} (see Lemma \ref{lm:asympvol} in the present paper). These  will be done in Section \ref{section3}.

In the Appendix, we give proofs of several facts used in Section \ref{section3}, 
and prove a comparison principle (see Theorem \ref{th:comparison principle}) for $\partial_t u=\Delta \phi(u)$
over general domains $\Omega$ including the case where $\partial\Omega$ is  unbounded (in this case we could not find a proof of Theorem \ref{th:comparison principle} in the literature). Once the comparison principle is proved, then the strong comparison principle follows immediately.


\setcounter{equation}{0}
\setcounter{theorem}{0}

\section{Short-time asymptotic profile in the unbounded case and application}
\label{section simple application}


We begin with our extension of formula \eqref{varadhan formula} to the case in which $\pa\Om$ 
is unbounded.

\begin{theorem}
\label{th:nonlinear varadhan} Let $\Om\subset\RE^N, \, N\ge 2,$ be any domain 
with boundary $\pa\Om$ of class $C^2$ and let $u$ be the solution of either problem 
\eqref{diffusion}-\eqref{initial} or \eqref{cauchy}.
\par
Then \eqref{varadhan formula} holds true.
\end{theorem}
\begin{remark}{\rm
\label{concern regularity Th 2.1} In view of the proof given below, instead of assuming that $\partial\Omega$ is of class $C^2$, we only need to assume that $\partial\Omega = \partial\left(\mathbb R^N\setminus \overline{\Omega}\right)$  under the existence of the solution $u$ of problem \eqref{diffusion}-\eqref{initial}. Of course, in the case of problem \eqref{cauchy}, we only need to assume that $\partial\Omega = \partial\left(\mathbb R^N\setminus \overline{\Omega}\right)$.
}
\end{remark}

\begin{proof}
The case where $\partial\Omega$ is bounded is treated in \cite{MSpoincare};
here, we shall assume that $\pa\Om$ is unbounded. 

Take any point $x_0 \in \Omega$. For each $\varepsilon > 0$, there exists an open ball $B_\delta(z),$ 
centered at $z$ and with radius $\delta,$ contained in $\mathbb R^N \setminus \ovr{\Om},$ and such that $|x_0-z| < d(x_0) + \varepsilon.$

Consider problem {\rm (\ref{diffusion})-(\ref{initial})} first. Let $u^\pm = u^\pm(x,t)$ be bounded solutions of the following initial-boundary value problems:
\begin{eqnarray*}
&\partial_t u^+=\Delta \phi(u^+)\ \ &\mbox{in }\ B_{d(x_0)}(x_0)\times (0,+\infty),\label{diffusion inside}\\
&u^+=1\ \ &\mbox{on }\ \partial B_{d(x_0)}(x_0)\times (0,+\infty),\label{dirichlet inside}\\
&u^+=0\ \ &\mbox{on }\ B_{d(x_0)}(x_0)\times \{0\},\label{initial-inside}
\end{eqnarray*}
and
\begin{eqnarray*}
&\partial_t u^-=\Delta \phi(u^-)\ \ &\mbox{in }\ \left(\mathbb R^N \setminus \overline{B_{\delta}(z)}\right)\times (0,+\infty),\label{diffusion outside}\\
&u^-=1\ \ &\mbox{on }\ \partial B_{\delta}(z)\times (0,+\infty),\label{dirichlet outside}\\
&u^-=0\ \ &\mbox{on }\ \left(\mathbb R^N \setminus \overline{B_{\delta}(z)}\right)\times \{0\},\label{initial outside}
\end{eqnarray*}
respectively.
Then it follows from the comparison principle that
\begin{equation}
\label{comparison from both sides}
u^-(x_0,t)\le u(x_0,t) \le u^+(x_0,t)\ \mbox{ for every }\ t>0,
\end{equation}
which gives
$$
-4t\Phi(u^-(x_0,t)) \ge -4t\Phi(u(x_0,t)) \ge -4t\Phi(u^+(x_0,t))\ \mbox{ for every }\ t>0.
$$
By \cite[Theorem 1.1]{MSpoincare}, letting $t \to 0^+$ yields that
$$
(d(x_0)+\varepsilon)^2 \ge \limsup_{t \to 0^+}\left(  -4t\Phi(u(x_0,t) \right) \ge \liminf_{t \to 0^+}\left(  -4t\Phi(u(x_0,t) \right) = d(x_0)^2.
$$
This implies \eqref{varadhan formula}, since $\veps>0$ is arbitrary. 
Furthermore, let $\rho_0$ and $\rho_1$ be given such that $0 < \rho_0 \le \rho_1 < +\infty;$ 
then by a scaling argument, we infer that the convergence in  \eqref{varadhan formula} is uniform
in every subset $F$ of $\{ x \in \Omega\ :\ \rho_0 \le d(x) \le \rho_1  \}$ in which
the number $\delta > 0$ can be chosen independently of each point $x \in F$. 
In particular, when $F$ is compact, it was shown in \cite[Lemma 3.11, p. 444]{Va} that $\delta > 0$ can be chosen independently of each point $x \in F$ only under the assumption that $\partial\Omega = \partial\left(\mathbb R^N\setminus \overline{\Omega}\right)$.

It remains to consider problem {\rm (\ref{cauchy})}. Let $u^\pm = u^\pm(x,t)$ be bounded solutions of the following initial value problems:
\begin{equation*}
\label{cauchy inside}
\partial_t u^+=\Delta \phi(u^+)\ \mbox{ in }\ \mathbb R^N \times (0, +\infty)\quad\mbox{ and }\ u^+ = \chi_{{B_{d(x_0)}(x_0)}^c}\ \mbox{ on }\ \mathbb R^N \times \{0\},
\end{equation*}
and
\begin{equation*}
\label{cauchy outside}
\partial_t u^-=\Delta \phi(u^-)\ \mbox{ in }\ \mathbb R^N \times (0, +\infty)\quad\mbox{ and }\ u^- = \chi_{\overline{B_{\delta}(z)}}\ \mbox{ on }\ \mathbb R^N \times \{0\},
\end{equation*}
respectively. Then by the comparison principle we get \eqref{comparison from both sides}. 
Thus, \eqref{varadhan formula}  follows similarly also in this case, with the aid of \cite[Theorem 4.1]{MSpoincare}.
\end{proof}
\vskip.2cm
We now give a simple application of the theorem just proved. Let $f \in C^2(\mathbb R^{N-1})$ and set
$$
\Omega = \{ x \in \mathbb R^N\ :\ x_N > f(x^\prime) \}, 
$$
where $x^\prime = (x_1,\cdots, x_{N-1}) \in \mathbb R^{N-1}$. Consider the solution $u=u(x,t)$  of either problem {\rm (\ref{diffusion})-(\ref{initial})} or problem {\rm (\ref{cauchy})}. 
In the sequel, it will be useful to know that
\begin{equation}
\label{decreasing in xN}
\frac {\partial u}{\partial x_N} < 0\ \mbox{either  in } \Omega \times (0, +\infty) \mbox{ or in } \mathbb R^N \times(0,+\infty);
\end{equation}
this is obtained by applying the comparison principle to $u(x^\prime, x_N+h, t)$ and $u(x, t)$ for $h>0$ 
and then the strong maximum principle to the resultant nonnegative function $\frac {\partial \phi(u)}{\partial x_N},$ since $v=\phi(u)$ satisfies $\partial_t v = \phi^\prime(u)\Delta v.$
\par
A hypersurface $\Gamma$ in $\Omega$ is said to be a {\it stationary level surface} of $u$ 
if at each time $t$ the solution $u$ remains constant on $\Gamma$ (a constant depending on $t$).
The following theorem characterizes the boundary $\partial\Omega$ 
in such a way that $u$ has a stationary level surface in $\Om.$


\begin{theorem}
\label{th:hyperplane characterization} Assume that for each $y^\prime \in \mathbb R^{N-1}$ there exists $h(y^\prime) \in \mathbb R$ such that
\begin{equation}
\label{BCNmodified}
\lim_{|x^\prime| \to \infty} \left[ f(x^\prime + y^\prime) - f(x^\prime) \right] = h(y^\prime).
\end{equation}
Let $u$ be the solution of either problem {\rm (\ref{diffusion})-(\ref{initial})} or problem {\rm (\ref{cauchy})}.
Suppose that $u$ has a stationary level surface $\Gamma$ in $\Omega$. 
\par
Then $f$ is affine, that is, $\partial\Omega$ must be a hyperplane.
\end{theorem}
\begin{remark}{\rm
\label{concern regularity Th 2.2} In view of the proof given below, instead of assuming that $f \in C^2(\mathbb R^{N-1})$, we only need to assume that $f \in C^0(\mathbb R^{N-1})$  under the existence of the solution $u$ of problem \eqref{diffusion}-\eqref{initial}. Of course, in the case of problem \eqref{cauchy}, we can replace the assumption $f \in C^2(\mathbb R^{N-1})$  with $f \in C^0(\mathbb R^{N-1})$.
}
\end{remark}

\begin{proof}
We shall use the sliding method due to Berestycki, Caffarelli, and Nirenberg \cite{BCN}. 
The condition \eqref{BCNmodified} is a modified version of (7.2) of \cite[p. 1108]{BCN}, 
in which $h(y^\prime)$ is supposed identically zero. 
\par
Since $\Gamma$ is a stationary level surface of $u$, 
it follows from Theorem \ref{th:nonlinear varadhan}, \eqref{decreasing in xN} and the implicit function theorem that 
there exist a number $R > 0$ and a function $g \in C^2(\mathbb R^{N-1})$ such that
\begin{eqnarray}
\Gamma =\{ (x^\prime,g(x^\prime))\in \mathbb R^N :\ x^\prime\in\mathbb R^{N-1}\} 
= \{ x \in \mathbb R^N :\ d(x) = R \};\label{representation of Gamma}
\end{eqnarray}
moreover, it is easy to verify that the function $g$ satisfies
\begin{equation}
g(x^\prime) = \sup_{|x^\prime-y^\prime| \le R} \{ f(y^\prime) + \sqrt{R^2 -|x^\prime-y^\prime|^2} \}\  \mbox{ for every } 
 x^\prime \in \mathbb R^{N-1}.\label{supconvolution}
\end{equation}

Conversely, let $\nu(y^\prime)$ denote the unit upward normal vector to $\Gamma$ at $(y^\prime,g(y^\prime)) \in \Gamma;$ the facts that $g$ is smooth, $\partial\Omega$ is a graph, and $(y^\prime,g(y^\prime)) - R\nu(y^\prime) \in \partial\Omega$ for every $y^\prime\in\mathbb R^{N-1}$ imply that
\begin{eqnarray}
f(x^\prime) &=& \inf_{|x^\prime-y^\prime| \le R} \{ g(y^\prime) -\sqrt{R^2 -|x^\prime-y^\prime|^2} \}\ \mbox{ for every } x^\prime \in \mathbb R^{N-1} ;\label{infconvolution}
\\
\partial\Omega &=& \{ x \in \mathbb R^N :\ \mbox{ dist}(x, \{ y \in \mathbb R^N :\ y_N\ge g(y^\prime)\}) = R \}.\label{representation of partial Omega}
\end{eqnarray}
Thus, it follows from \eqref{representation of Gamma} and \eqref{representation of partial Omega} that for every $x \in \partial\Omega$ there exists $z \in \Gamma$ satisfying 
\begin{equation}
\label{touching balls}
B_R(z) \subset \Omega\ \mbox{ and } \partial B_R(z) \cap \partial\Omega = \{ x \}.
\end{equation}
\par
For fixed $y^\prime \in \mathbb R^{N-1}$ and $h \in \mathbb R$, we define the translates:
\begin{eqnarray*}
\Omega_{y^\prime,h} = (y', h)+\Om, \ \ \Gamma_{y^\prime,h} = (y', h)+\Gamma;
\end{eqnarray*}
\eqref{BCNmodified} guarantees that the values 
\begin{eqnarray}
\label{upper and lower approximations}
&&h_+(y')=\inf\{ h \in \mathbb R : \Omega_{y^\prime,h}\subset\Omega\} \ \mbox{ and } \nonumber\\
&&h_-(y')=\sup\{ h \in \mathbb R : \Omega \subset \Omega_{y^\prime,h} \} 
\end{eqnarray}
are finite, since in fact, $h_-(y') \le h(y^\prime)\le h_+(y')$ for every  $y^\prime\in\mathbb R^{N-1}.$
\par
To complete our proof, it suffices to show that 
$$
h_-(y') = h(y^\prime) = h_+(y').
$$
Indeed, this yields that $\Omega = \Omega_{y^\prime,h(y^\prime)}$ for every $y^\prime \in \mathbb R^{N-1}$
and hence 
\begin{equation}
\label{functional equation}
f(x^\prime) = f(x^\prime-y^\prime) + h(y^\prime)\ \mbox{ for every } x^\prime, y^\prime \in \mathbb R^{N-1}. 
\end{equation}
Then, $\nabla f(x^\prime) = \nabla f(x^\prime-y^\prime)$ for every $x^\prime, y^\prime \in \mathbb R^{N-1}$ and hence $\nabla f $ must be constant in $\mathbb R^{N-1}$. Namely, $f$ is affine and $\partial\Omega$ must be a hyperplane. When it is assumed only that $f \in C^0(\mathbb R^{N-1})$, without using differentiability of $f$, we can solve \eqref{functional equation} as a functional equation  with the help of continuity of $f$ and we can also  conclude that $f $ is affine.
\par
Thus, set $h_+ = h_+(y^\prime)$ and suppose by contradiction that $h_+ > h(y^\prime)$. Then it follows from \eqref{BCNmodified} and \eqref{touching balls} that there exist $x_0 \in \partial\Omega\cap\partial\Omega_{y^\prime, h_+}$ and $z \in \Gamma \cap \Gamma_{y^\prime,h_+}$ satisfying
$$
\Omega_{y^\prime,h_+} \subsetneqq  \Omega\ \mbox{ and  }\  \partial B_R(z) \cap \partial\Omega\cap\partial\Omega_{y^\prime, h_+} = \{x_0\}.
$$
On the other hand, from the strong comparison principle we have
$$
u(x^\prime-y^\prime,x_N-h_+, t) > u(x,t)\ \mbox{ for every } (x,t) \in \Omega_{y^\prime,h_+}\times (0,\infty).
$$
Therefore, $u(z^\prime-y^\prime,x_N-h_+, t) > u(z,t)$ which contradicts the fact that $z \in \Gamma \cap \Gamma_{y^\prime,h_+}$ and that $\Gamma$ is a stationary level surface of $u$. 
\par
The proof that $h_-(y') = h(y^\prime)$ runs similarly. \end{proof}


\setcounter{equation}{0}
\setcounter{theorem}{0}

\section{Short-time asymptotics and curvature}
\label{section3}
This section is devoted to the proof of Theorem \ref{th:interaction curvatures}. 
We first prove two lemmas in which we construct useful barriers for
problems \eqref{diffusion}-\eqref{initial} and \eqref{cauchy}, respectively.

In the former lemma, we use a result from Atkinson and Peletier \cite{AtP}:  
for every $c > 0$, there exists a unique $C^2$ solution $f_c = f_c(\xi)$ of the problem:
\begin{eqnarray}
&& \left( \phi^\prime(f_c) f_c^\prime\right)^\prime + \frac 12 \xi  f_c^\prime = 0\ \mbox{ in }\ [0,+\infty),
\label{ode selfsimilar}
\\
&& f_c(0) = c,\quad f_c(\xi) \to 0\ \mbox{ as }\ \xi \to +\infty,\label{boundary conditions}
\\
&& f_c^\prime < 0\ \mbox{ in }\ [0,+\infty).
\label{monotonicity}
\end{eqnarray}
Note that, if we put $w(s,t) = f_c\left(t^{-\frac 12}s\right)$ for $ s >0$ and $t > 0$, then $w$ satisfies the one-dimensional problem:
$$
\partial_t w = \partial_s^2 \phi(w)\ \mbox{ in } (0,+\infty)^2,\ w = c\ \mbox{ on }\ \{0\} \times (0,+\infty), \mbox{ and } w = 0\ \mbox{ on } (0,+\infty)\times \{0\}.
$$

\begin{lemma}
\label{barriers for IBVP}
Let $\pa\Om$ be bounded and of class $C^2$ and let $\rho_0>0$ be such that the distance function
$d$ belongs to $C^2(\ovr{\Om_{\rho_0}})$ (see \cite{GT}); then, set $\rho_1=\max\{ 2R, \rho_0\}.$
Let $u=u(x,t)$ be  the solution of problem \eqref{diffusion}-\eqref{initial}.
\par 
Then, for every $\veps\in(0,1/4),$ there exist two $C^2$ functions $f_\pm= f_\pm(\xi) :[0,+\infty)\to\RE$ satisfying
\begin{eqnarray}
&&0 < f_\pm(\xi) \le \alpha e^{-\beta\xi^2}\ \mbox{ for every }\xi \in [0,+\infty);\label{exponential decay1}
\\
&&  f_\pm \to f_1 \mbox{ as } \varepsilon \to 0 \mbox{ uniformly on } [0,+\infty),\label{converge to f1}
\end{eqnarray}
where  $\alpha$ and $\beta$ are positive constants independent of $\varepsilon$, and there exists a number $\tau=\tau_\veps>0$ such that
the functions $w_\pm,$ defined by
\begin{equation}
\label{upper and lower solutions}
w_\pm(x,t) = f_\pm\left(t^{-\frac 12}d(x)\right)\ \mbox{ for } (x,t) \in \Omega \times (0,+\infty),
\end{equation}
satisfy the inequalities:
\begin{equation}
\label{important inequalities}
w_- \le u\le w_+\ \mbox{ in }\ \overline{\Om_{\rho_1}}\times (0,\tau].
\end{equation}
\end{lemma}
\begin{proof}
We begin by deriving some properties of the solution $f_c$ 
of problem \eqref{ode selfsimilar}-\eqref{monotonicity}; 
by writing $v_c = v_c(\xi) = \phi\left(f_c(\xi)\right)$ for $\xi \in [0,+\infty)$, we see that
\begin{equation}
\label{equation of v}
\frac {v_c^{\prime\prime}}{v_c^\prime} = - \frac 12 \xi \frac 1{\phi^\prime(f_c)}\ \mbox{ in }\ [0,+\infty).
\end{equation}
\par
With the aid of the last assumption in \eqref{nonlinearity}, integrating \eqref{equation of v} yields that
\begin{equation}
\label{estimate of v prime}
v_c^\prime(0) \exp\left\{ -\frac { \xi^2}{4\delta_2}\right\} \le v_c^\prime(\xi) \le v_c^\prime(0) \exp\left\{ -\frac { \xi^2}{4\delta_1}\right\} < 0\ \mbox{ for every } \xi > 0,
\end{equation}
and hence
\begin{equation}
\label{estimate of f prime}
\frac{v_c^\prime(0)}{\delta_1} \exp\left\{ -\frac { \xi^2}{4\delta_2}\right\} \le f_c^\prime(\xi) \le \frac{v_c^\prime(0)}{\delta_2} \exp\left\{ -\frac { \xi^2}{4\delta_1}\right\} < 0\ \mbox{ for every } \xi > 0.
\end{equation}
Furthermore, by integrating \eqref{estimate of f prime} and using \eqref{boundary conditions}, we have that for every $\xi > 0$
\begin{equation}
\label{estimate of f }
-\frac{v_c^\prime(0)}{\delta_1} \int_\xi^\infty\exp\left\{ -\frac { \eta^2}{4\delta_2}\right\} d\eta\ge f_c(\xi) \ge -\frac{v_c^\prime(0)}{\delta_2}\int_\xi^\infty \exp\left\{ -\frac { \eta^2}{4\delta_1}\right\} d\eta.
\end{equation}
Thus, with the aid of \eqref{estimate of v prime} and \eqref{estimate of f }, by integrating \eqref{ode selfsimilar}, we have:
\begin{equation}
 -v_c^\prime(0) = \frac 12\int_0^\infty f_c(\xi)\ d\xi\ \mbox{ for } \ c > 0.\label{integral formula for derivative of v}
 \end{equation}
Moreover,  a comparison argument will give us
\begin{eqnarray}
&&0 <  f_{c_1} < f_{c_2}\ \mbox{  on } \ [0,+\infty) \ \mbox{ if }\ 0 < c_1 < c_2 < +\infty;\label{monotonicity of f with parameter}
\\
&& 0 > v_{c_1}^\prime(0) >  v_{c_2}^\prime(0) \ \mbox{ if }\ 0 < c_1 < c_2 < +\infty. \label{monotonicity of derivatives of v}
\end{eqnarray}
In the Appendix, we will give a proof of \eqref{integral formula for derivative of v}-\eqref{monotonicity of derivatives of v}.
\par
Furthermore,  \cite[Lemma 4, p. 383]{AtP} tells us that, for every compact interval $I$ contained in  $(0, +\infty)$,  
\begin{equation}
\label{limiting behavior}
\frac {-4\Phi(f_c(\xi))}{\xi^2} \to 1\ \mbox{ as }\ \xi \to +\infty\ \mbox{ uniformly for }\ c \in I.
\end{equation}

Let $0< \varepsilon < \frac 14$. Then, by continuity 
we can find a sufficiently small $0 < \eta_\varepsilon << \varepsilon$ and two $C^2$ functions $f_{\pm} = f_{\pm}(\xi)$ for $\xi \ge 0$ satisfying:
\begin{eqnarray*}
&& f_\pm(\xi) = f_{1\pm \varepsilon}\left(\sqrt{1\mp 2\eta_\varepsilon}\ \xi\right)\ \mbox{ if } \xi \ge \eta_\varepsilon; \label{scaling of f}
\\
&&f_\pm^\prime < 0\ \mbox{ in } [0,+\infty);
\label{monotonicity both}
\\
&& f_- < f_1 < f_+\ \mbox{ in } [0,+\infty);
\label{from above and below}
\\
&& \left(\phi^\prime(f_\pm)f_\pm^\prime\right)^\prime + \frac 12\xi f_\pm^\prime = h_\pm(\xi) f_\pm^\prime\ \mbox{ in } [0,+\infty),\label{ modified ode}
\end{eqnarray*}
where $h_\pm = h_\pm(\xi)$ are defined by
\begin{equation*}
\label{definition of h}
h_\pm(\xi) = \left\{\begin{array}{rl}   \pm \eta_\varepsilon\xi\ &\mbox{ if }\ \xi \ge \eta_\varepsilon,\\
  \pm\eta_\varepsilon^2\ &\mbox{ if }\ \xi \le \eta_\varepsilon.
\end{array}\right.
\end{equation*}
(Here, in order to use the functions $h_\pm$ also in Lemma \ref{barriers for Cauchy} later, we defined $h_\pm(\xi)$ for all $\xi \in \mathbb R$.)
It is important to notice that 
\begin{equation}
\label{useful properties}
h_+=-h_- \ge \eta_\varepsilon^2\  \mbox{ on } \mathbb R.
\end{equation}
Moreover, \eqref{converge to f1} follows directly from the above construction of $f_\pm$, and \eqref{estimate of f } together with \eqref{monotonicity of f with parameter} yields  \eqref{exponential decay1}.

Set $\Psi = \Phi^{-1}$. Then it follows from \eqref{limiting behavior} that there exists $\xi_\varepsilon > 1$ such that
\begin{equation}
\label{inequality involving Psi}
\Psi\left(-\frac {\xi^2}4\left(1-\frac{\eta_\varepsilon}2\right)\right) > f_c(\xi) > \Psi\left(-\frac {\xi^2}4\left(1+\frac {\eta_\varepsilon}2\right)\right) \ \mbox{ for } \xi \ge \xi_\varepsilon \mbox{ and } c \in I_\varepsilon,
\end{equation}
where we set $I_\varepsilon=[1-2\varepsilon, 1+2\varepsilon]$. 

Since $\partial\Omega$ is bounded and of class $C^2$, 
Theorem \ref{th:nonlinear varadhan} yields that
\begin{equation}
\label{application of varadhan}
-4t\Phi(u(x,t)) \to d(x)^2\ \mbox{ as } t \to 0^+\ \mbox{ uniformly on } \overline{\Omega_{\rho_1}}\setminus \Omega_{\rho_0}.
\end{equation}
Then there exists $\tau_{1,\varepsilon} > 0$ such that for every $t \in (0,\tau_{1,\varepsilon}]$ and every $x \in \overline{\Omega_{\rho_1}}\setminus \Omega_{\rho_0}$
$$
\left| -4t\Phi(u(x,t)) - d(x)^2\right| < \frac 12\eta_\varepsilon\rho_0^2 \le \frac 12\eta_\varepsilon d(x)^2,
$$
which implies that 
\begin{equation}
\label{inequality inside}
\Psi\left(-\frac {\left(1-\frac 12\eta_\varepsilon\right)}4 \frac {d(x)^2}t \right)> u(x,t) > \Psi\left(-\frac {\left(1+\frac 12\eta_\varepsilon\right)}4 \frac {d(x)^2}t \right),
\end{equation}
for every $t \in (0,\tau_{1,\varepsilon}]$ and every $x \in \overline{\Omega_{\rho_1}}\setminus \Omega_{\rho_0}.$
\par
From \eqref{inequality involving Psi}, we have
\begin{eqnarray}
&&f_+(\xi) = f_{1+\varepsilon}(\sqrt{1-2\eta_\varepsilon}\ \xi) > \Psi\left(-\frac{\xi^2}4\left(1-\frac{\eta_\varepsilon}2\right)\right) \mbox{ if } \xi \ge \frac {\xi_\varepsilon}{\sqrt{1-2\eta_\varepsilon}};\label{f plus}
\\
&&f_-(\xi) = f_{1-\varepsilon}(\sqrt{1+2\eta_\varepsilon}\ \xi) < \Psi\left(-\frac{\xi^2}4\left(1+\frac{\eta_\varepsilon}2\right)\right) \mbox{ if } \xi \ge \frac {\xi_\varepsilon}{\sqrt{1+2\eta_\varepsilon}}.\label{f minus}
\end{eqnarray}
\par
Now, consider the two functions $w_\pm = w_\pm(x,t)$ defined by \eqref{upper and lower solutions}.
It follows from \eqref{inequality inside}, \eqref{f plus} and \eqref{f minus} that there exists $\tau_{2,\varepsilon} \in (0, \tau_{1,\varepsilon}]$ satisfying
\begin{equation}
\label{estimate inside}
w_- < u < w_+\ \mbox{ in } \left(\overline{\Omega_{\rho_1}}\setminus\Omega_{\rho_0}\right) \times (0,\tau_{2,\varepsilon}].
\end{equation}
Since $d \in C^2(\overline{\Omega_{\rho_0}})$ and $|\nabla d| = 1$ in $\overline{\Omega_{\rho_0}}$, we have
\begin{equation*}
\label{application of parabolic operator}
\partial_t w_\pm-\Delta\phi(w_\pm) = -f_\pm^\prime t^{-1}\left\{ h_\pm+ \sqrt{t}\ \phi^\prime(f_\pm)\Delta d\right\}\ \mbox{ in }\ \overline{\Omega_{\rho_0}} \times (0, +\infty).
\end{equation*}
Therefore, it follows from  \eqref{useful properties} that there exists $\tau_{3,\varepsilon} \in (0, \tau_{2,\varepsilon}]$ satisfying
\begin{equation*}
\label{differential inequalities}
\partial_t w_- - \Delta\phi(w_-) < 0 < \partial_t w_+ - \Delta\phi(w_+)\ \mbox{ in }\ \Omega_{\rho_0}\times (0,\tau_{3,\varepsilon}].
\end{equation*}
Observe that
\begin{eqnarray*}
&& w_-=u=w_+ = 0\ \mbox{ in }\ \Omega_{\rho_0} \times \{0\},\label{comparison initial}
\\
&& w_- =f_-(0) < 1=f_1(0) =u < f_+(0)=w_+\ \mbox{ on }\ \partial\Omega \times (0,\tau_{3,\varepsilon}],\label{comparison outside boundary}
\\
&&w_- < u < w_+\ \mbox{ on }\ \Gamma_{\rho_0} \times (0,\tau_{3,\varepsilon}].\label{comparison inside boundary}
\end{eqnarray*}
Note that the last inequalities above come from \eqref{estimate inside}. 
\par
Thus, \eqref{important inequalities} holds true with $\tau=\tau_{3,\varepsilon},$ 
by the comparison principle and \eqref{estimate inside}.
\end{proof}
\vskip.2cm
In the next lemma, instead of \eqref{ode selfsimilar}-\eqref{monotonicity}, we will work with 
the following problem:
\begin{eqnarray}
&& \left( \phi^\prime(f_c) f_c^\prime\right)^\prime + \frac 12 \xi  f_c^\prime = 0\ \mbox{ in }\ \mathbb R,\label{ode selfsimilar on whole line}
\\
&& f_c(\xi) \to c\ \mbox{ as }\ \xi \to -\infty,\quad f_c(\xi) \to 0\ \mbox{ as }\ \xi \to +\infty,\label{boundary conditions on whole line}
\\
&& f_c^\prime < 0\ \mbox{ in }\ \mathbb R.\label{monotonicity on whole line}
\end{eqnarray}
In the Appendix we will prove that, for every $c > 0$, 
\eqref{ode selfsimilar on whole line}-\eqref{monotonicity on whole line} has 
a unique $C^2$ solution $f_c = f_c(\xi).$ 
Note that, if we put $w(s,t) = f_c\left(t^{-\frac 12}s\right)$ for $ s \in \mathbb R$ and $t > 0$, then $w$ satisfies the one-dimensional  initial value problem:
$$
\partial_t w = \partial_s^2 \phi(w)\ \mbox{ in } \mathbb R \times (0,+\infty) \ \mbox{ and }\ w = c \chi_{(-\infty,0]}\ \mbox{ on } \mathbb R\times\{0\}.
$$
\par
Also, let us consider the signed distance function $d^*= d^*(x)$ 
of $x \in \mathbb R^N$ to the boundary $\partial\Omega$ defined by
\begin{equation*}
\label{signed distance}
d^*(x) = \left\{\begin{array}{rll}
 \mbox{ dist}(x,\partial\Omega)\ &\mbox{ if }\ x \in \Omega,
\\
-\mbox{ dist}(x,\partial\Omega)\ &\mbox{ if  }\ x \not\in \Omega.
\end{array}\right.
\end{equation*}
If $\partial\Omega$  is bounded and of class $C^2$,  there exists a number $\rho_0 > 0$ such that $d^*(x)$ is $C^2$-smooth on a compact neighborhood  $\mathcal N$ of the boundary $\partial\Omega$ given by
\begin{equation*}
\label{neighborhood of boundary from both sides}
\mathcal N = \{ x \in \mathbb R^N : -\rho_0 \le d^*(x) \le \rho_0 \}.
\end{equation*}
For simplicity we have used the same letter $\rho_0 >0$ as in  Lemma \ref{barriers for IBVP}.

\begin{lemma}
\label{barriers for Cauchy}
Let $\pa\Om$ be bounded and of class $C^2,$ set $\rho_1=\max\{2R, \rho_0\}$
and let $u=u(x,t)$ be  the solution of problem  {\rm (\ref{cauchy})}.
\par 
Then, for every $\veps\in(0,1/4),$ there exist  two $C^2$ functions $f_\pm= f_\pm(\xi) : \mathbb R \to\RE$ satisfying
\begin{eqnarray}
&&0 < f_\pm(\xi) \le \alpha e^{-\beta\xi^2}\ \mbox{ for every }\xi \in [0,+\infty);\label{exponential decay2}
\\
&&  f_\pm \to f_1 \mbox{ as } \varepsilon \to 0 \mbox{ uniformly on } [0,+\infty),\label{converge to f12}
\end{eqnarray}
where  $\alpha$ and $\beta$ are positive constants independent of $\varepsilon$, and there exists a number $\tau=\tau_\veps>0$ such that
the functions $w_\pm,$ defined by
\begin{equation}
\label{sub and super solution Cauchy}
w_\pm(x,t) = f_\pm\left(t^{-\frac12} d^*(x)\right)\ \mbox{ for }\  (x,t) \in \mathbb R^N \times (0,+\infty),
\end{equation}
satisfy the inequalities:
\begin{equation}
\label{important inequalities cauchy}
w_- \le u\le w_+\ \mbox{ in }\ \overline{\mathcal N \cup \Omega_{\rho_1}}\times (0,\tau].
\end{equation}
\end{lemma}
\begin{proof}
Let $f_c$ be the solution of problem \eqref{ode selfsimilar on whole line}-\eqref{monotonicity on whole line}; by writing $v_c = v_c(\xi) = \phi\left(f_c(\xi)\right)$ for $\xi \in \mathbb R$, we have:
\begin{eqnarray}
&& -v_c^\prime(0) = \frac 12\int_0^\infty f_c(\xi)\ d\xi\ \mbox{ for } \ c > 0;\label{integral formula for derivative of v cauchy}
\\
&& 0 <  f_{c_1} < f_{c_2}\ \mbox{  in } \ \mathbb R \ \mbox{ if }\ 0 < c_1 < c_2 < +\infty;\label{monotonicity of f with parameter cauchy}
\\
&& 0 > v_{c_1}^\prime(0) >  v_{c_2}^\prime(0) \ \mbox{ if }\ 0 < c_1 < c_2 < +\infty. \label{ monotonicity of derivatives of v cauchy}
\end{eqnarray}
In the Appendix we will give a proof of \eqref{integral formula for derivative of v cauchy}-\eqref{ monotonicity of derivatives of v cauchy}.
Then \cite[Lemma 4, p. 383]{AtP} tells us that \eqref{limiting behavior} also holds for the solution $f_c$ of this problem.

Let $0 < \varepsilon < \frac 14$. 
We can find a sufficiently small $0 < \eta_\varepsilon << \varepsilon$ and two $C^2$ functions $f_{\pm} = f_{\pm}(\xi)$ for $\xi  \in \mathbb R$ satisfying:
\begin{eqnarray}
&& f_\pm(\xi) = f_{1\pm\varepsilon} \left(\sqrt{1\mp 2\eta_\varepsilon}\ \xi\right)\ \mbox{ if } \xi \ge \eta_\varepsilon; \label{scaling of f on whole line}
\\
&&f_\pm^\prime < 0\ \mbox{ in } \mathbb R;\label{monotonicity both on whole line}
\\
&& f_-(-\infty) < 1= f_1(-\infty) < f_+(-\infty)\ \mbox{ and }\ f_- < f_1< f_+\ \mbox{ in } \mathbb R;\label{from above and below on whole line}
\\
&& \left(\phi^\prime(f_\pm)f_\pm^\prime\right)^\prime + \frac 12\xi f_\pm^\prime = h_\pm(\xi) f_\pm^\prime\ \mbox{ in } \mathbb R.\label{ modified ode on whole line}
\end{eqnarray}
In the Appendix we will prove \eqref{from above and below on whole line} by choosing $\eta_\varepsilon>0$ sufficiently small. 

Here, we also have \eqref{useful properties}, \eqref{exponential decay2}, and \eqref{converge to f12}.
Moreover,  
it follows from \eqref{limiting behavior} that there exists $\xi_\varepsilon > 1$ satisfying \eqref{inequality involving Psi}.
Proceeding similarly yields  \eqref{application of varadhan}, \eqref{inequality inside}, \eqref{f plus} and \eqref{f minus}. 
\par
Now, consider the functions $w_\pm$ defined by \eqref{sub and super solution Cauchy}.
Then we also have \eqref{estimate inside} and, 
since $d^* \in C^2(\mathcal N)$ and $|\nabla d^*| = 1$ in $\mathcal N$, we obtain that
\begin{equation*}
\label{application of parabolic operator cauchy}
\partial_t w_\pm-\Delta\phi(w_\pm) = -f_\pm^\prime t^{-1}\left\{ h_\pm+ \sqrt{t}\ \phi^\prime(f_\pm)\Delta d^*\right\}\ \mbox{ in }\ \mathcal N \times (0, +\infty).
\end{equation*}
Therefore, it follows from  \eqref{useful properties} that there exists $\tau_{3,\varepsilon} \in (0, \tau_{2,\varepsilon}]$ satisfying:
\begin{eqnarray*}
&&\partial_t w_- - \Delta\phi(w_-) < 0 < \partial_t w_+ - \Delta\phi(w_+)\ \mbox{ in }\ \mathcal N\times (0,\tau_{3,\varepsilon}],\label{differential inequalities cauchy}
\\
&& w_-\le u \le w_+ \ \mbox{ in }\ \mathcal N \times \{0\},\label{comparison initial cauchy}
\\
&& w_- < u < w_+\ \mbox{ on }\ \partial\mathcal N \times (0,\tau_{3,\varepsilon}].\label{comparison outside boundary cauchy} 
\end{eqnarray*}
Note that, in the last inequalities, the ones on $\Gamma_{\rho_0} \times (0,\tau_{3,\varepsilon}]$  come from \eqref{estimate inside} and the others on $\left(\partial\mathcal N \setminus \Gamma_{\rho_0}\right)\times (0,\tau_{3,\varepsilon}]$ come from the former formula of \eqref{from above and below on whole line}. 
\par
Thus, \eqref{important inequalities cauchy} follows, with $\tau=\tau_{3,\veps},$ from the comparison principle and \eqref{estimate inside}.
\end{proof}
\vskip.2cm

In the proof of Theorem \ref{th:interaction curvatures}, 
we will also use a geometric lemma from \cite{MSprsea} adjusted to our situation.

\begin{lemma}{\rm (\cite[Lemma 2.1, p. 376]{MSprsea})}
\label{lm:asympvol}
Let $\kappa_j(y_0)<\frac1{R}$ for every $j=1,\dots,N-1.$ Then we have:
\begin{equation*}
\label{asympvol}
\lim_{s\to 0^+} s^{-\frac{N-1}{2}} \mathcal H^{N-1}(\Gamma_s\cap B_R(x_0))=2^{\frac{N-1}{2}}\omega_{N-1} \
\left\{\prod_{j=1}^{N-1}\left(\frac1{R}-\kappa_j(y_0)\right)\right\}^{-\frac12},
\end{equation*}
where $\mathcal H^{N-1}$ is the standard $(N-1)$-dimensional Hausdorff measure, and $\omega_{N-1}$ is the volume of the unit ball in $\mathbb R^{N-1}.$
\end{lemma}

\begin{proofC}
We distinguish two cases: 
$$
{\rm (I)}\  \partial\Omega\ \mbox{ is bounded and of class $C^2$;}\quad {\rm (II)}\  \partial\Omega\ \mbox{ is otherwise.}
$$
\par
Let us first show how we obtain case (II) once we have proved case (I). 
Indeed, we can find two $C^2$ domains, say $\Omega_1$ and $\Omega_2,$ with bounded boundaries, 
and a ball $B_\delta(y_0)$ with the following properties:
$\Omega_1$ and $\mathbb R^N \setminus \overline{\Omega_2}$ are bounded; 
$B_R(x_0)\subset\Omega_1 \subset \Omega \subset \Omega_2;$
\begin{equation*}
\label{two constructed domains}
B_\delta(y_0) \cap \partial\Omega \subset \partial\Omega_1\cap\partial\Omega_2\ 
\mbox{ and }\ \overline{B_R(x_0)} \cap \left(\mathbb R^N\setminus\Omega_i \right)= \{y_0\}\ \mbox{ for } i=1, 2.
 \end{equation*}
 \par
Let $u_i=u_i(x,t)\ (i=1,2)$ be the two bounded solutions of either problem \eqref{diffusion}-\eqref{initial} or problem \eqref{cauchy} where $\Omega$ is replaced by $\Omega_1$ or $\Omega_2$, respectively.
Since $\Omega_1 \subset \Omega \subset \Omega_2$, it follows from the comparison principle that
$$
u_2 \le u \ \mbox{ in }\ \Omega \times (0,+\infty) \ \mbox{ and }\ u \le u_1\ \mbox{ in }\ \Omega_1\times(0,+\infty).
$$
Therefore, it follows that for every $t > 0$
$$
t^{-\frac{N+1}4 }\!\!\!\int\limits_{B_R(x_0)}\! u_2(x,t)\ dx \le t^{-\frac{N+1}4 }\!\!\!\int\limits_{B_R(x_0)}\! u(x,t)\ dx \le t^{-\frac{N+1}4 }\!\!\!\int\limits_{B_R(x_0)}\! u_1(x,t)\ dx.
$$
These two inequalities show that case (I) implies case (II).
\par
Now, let us consider case (I).  
First, we take care of problem  \eqref{diffusion}-\eqref{initial}. 
Lemma \ref{barriers for IBVP} implies that for every $t \in (0,\tau]$
\begin{equation}
\label{estimates from both sides}
t^{-\frac{N+1}4 }\!\!\!\int\limits_{B_R(x_0)} w_- \ dx \le t^{-\frac{N+1}4 }\!\!\! \int\limits_{B_R(x_0)} u \ dx \le t^{-\frac{N+1}4 }\!\!\!\int\limits_{B_R(x_0)} w_+ \ dx.
\end{equation}
Also, with the aid of the co-area formula, we have:
$$
\int\limits_{B_R(x_0)} w_\pm\ dx = t^{\frac{N+1}4}\int_0^{2Rt^{-\frac 12}} f_\pm(\xi)\xi^{\frac {N-1}2} \left(t^{\frac 12}\xi\right)^{-\frac {N-1}2} \mathcal H^{N-1}\left(\Gamma_{t^{\frac12}\xi}\cap B_R(x_0)\right) d\xi.
$$
Thus, when $\kappa_j(y_0)<\frac1{R}$ for every $j=1,\dots,N-1,$ 
by Lebesgue's dominated convergence theorem, \eqref{exponential decay1}, and Lemma \ref{lm:asympvol},  we get
$$
\lim_{t\to 0^+} t^{-\frac{N+1}4}\!\!\!\int\limits_{B_R(x_0)} w_\pm\ dx = 2^{\frac{N-1}2}\omega_{N-1}\left\{\prod_{j=1}^{N-1}\left(\frac 1R-\kappa_j(y_0)\right)\right\}^{-\frac 12}\int_0^\infty f_\pm(\xi) \xi^{\frac{N-1}2} d\xi.
$$
Moreover, again by Lebesgue's dominated convergence theorem, \eqref{exponential decay1}, and \eqref{converge to f1}, we see  that
$$
\lim_{\varepsilon \to 0}\int_0^\infty f_\pm(\xi) \xi^{\frac{N-1}2} d\xi = \int_0^\infty f_1(\xi) \xi^{\frac{N-1}2} d\xi.
$$
Therefore, since $\varepsilon > 0$ is arbitrarily small in \eqref{estimates from both sides},  it follows that \eqref{asymptotics and curvatures} holds true, where we set
$$
c(\phi,N) = 2^{\frac{N-1}2}\omega_{N-1}\int_0^\infty f_1(\xi) \xi^{\frac{N-1}2} d\xi.
$$
\par
It remains to consider the case where $\kappa_j(y_0) = \frac 1R$ for some $j \in \{ 1, \cdots, N-1\}$. Choose a sequence of balls $\{ B_{R_k}(x_k) \}_{k=1}^\infty$  satisfying:
$$
 R_k < R,\ y_0 \in \partial B_{R_k}(x_k), \mbox{ and } B_{R_k}(x_k) \subset B_R(x_0)\mbox{ for every } k \ge 1,\ \mbox{ and }\ \lim_{k\to \infty}R_k = R.
$$
Since $\kappa_j(y_0)\le \frac 1R<\frac1{R_k}$ for every $j=1,\dots,N-1$ and every $k \ge 1$, we can apply the previous case to each $B_{R_k}(x_k)$ to see
 that for every $k \ge 1$
\begin{eqnarray*}
\liminf_{t \to 0^+}t^{-\frac{N+1}4 }\!\!\!\int\limits_{B_R(x_0)} u(x,t)\ dx &\ge& \liminf_{t \to 0^+}t^{-\frac{N+1}4 }\!\!\!\int\limits_{B_{R_k}(x_k)} u(x,t)\ dx 
\\
&=& c(\phi,N)\left\{\prod\limits_{j=1}^{N-1}\left(\frac 1{R_k} - \kappa_j(y_0)\right)\right\}^{-\frac 12}.
\end{eqnarray*}
Hence, letting $k \to \infty$ yields that 
$$
\liminf_{t \to 0^+}t^{-\frac{N+1}4 }\!\!\!\int\limits_{B_R(x_0)} u(x,t)\ dx = +\infty,
$$
which completes the proof for problem  \eqref{diffusion}-\eqref{initial}.
\par
The proof of \eqref{asymptotics and curvatures} in the case of problem \eqref{cauchy}
runs similarly with the aid of Lemmas \ref{barriers for Cauchy} and \ref{lm:asympvol}.
\end{proofC}


\setcounter{equation}{0}
\setcounter{theorem}{0}

\def\theequation{A.\arabic{equation}}
\def\thetheorem{A.\arabic{theorem}}

\appendix
\section*{Appendix}
\label{appendix}
Here, for the reader's convenience, we give proofs of several facts used in Section \ref{section3}, and prove a comparison principle (Theorem \ref{th:comparison principle}) for $\partial_t u=\Delta \phi(u)$
over general domains $\Omega$ including the case where $\partial\Omega$ is  unbounded.

\vskip 2ex
\noindent
{\bf Proof of {\rm \eqref{integral formula for derivative of v}-\eqref{monotonicity of derivatives of v}.}}
First of all,  \eqref{integral formula for derivative of v} and \eqref{monotonicity of f with parameter} imply \eqref{monotonicity of derivatives of v}.
It suffices to prove  \eqref{integral formula for derivative of v} and \eqref{monotonicity of f with parameter}. Let $c > 0.$ 
By integrating equation \eqref{ode selfsimilar} on $[0, \eta]$ for every $\eta > 0$ and integrating by parts, we get
\begin{equation*}
\label{integrating by parts 1}
v_c^\prime(\eta)-v_c^\prime(0) +\frac 12\eta f_c(\eta)-\frac 12\int_0^\eta f_c(\xi)\ d\xi=0.
\end{equation*}
Then, with the aid of \eqref{estimate of v prime} and \eqref{estimate of f }, letting $\eta \to \infty$ yields \eqref{integral formula for derivative of v}.

Let $0 < c_1 < c_2 < +\infty$. Since $f_{c_1}(0) = c_1 < c_2 = f_{c_2}(0)$, suppose that there exists $\xi_0 > 0$ satisfying
$$
f_{c_1}(\xi_0) = f_{c_2}(\xi_0)\ \mbox{ and }\ f_{c_1}(\xi) < f_{c_2}(\xi)\ \mbox{ for every } \xi \in [0, \xi_0).
$$
Then it follows from the uniqueness of  solutions of Cauchy problems for ordinary differential equations that 
\begin{equation}
\label{non tangential 0}
v_{c_2}^\prime(\xi_0) < v_{c_1}^\prime(\xi_0) < 0.
\end{equation}
Thus, we distinguish two cases:

\noindent
 (i)\ There exists $\xi_1 \in (\xi_0, \infty)$ satisfying
 $$
f_{c_1}(\xi_1) = f_{c_2}(\xi_1)\ \mbox{ and }\ f_{c_1}(\xi) > f_{c_2}(\xi)\ \mbox{ for every } \xi \in (\xi_0, \xi_1);
$$

\noindent
(ii)\ For every $\xi \in (\xi_0,\infty)$, $f_{c_1}(\xi) > f_{c_2}(\xi).$

\vskip 1ex
\noindent
In case (i), by the uniqueness, we also have
\begin{equation}
\label{non tangential 1}
v_{c_1}^\prime(\xi_1) < v_{c_2}^\prime(\xi_1) < 0.
\end{equation}
By integrating equation \eqref{ode selfsimilar} on $[\xi_0, \xi_1]$ for $f_{c_1}$ and $f_{c_2}$ and integrating by parts, we see that for $j =1, 2$
\begin{equation*}
\label{integrating by parts 2}
v_{c_j}^\prime(\xi_1)-v_{c_j}^\prime(\xi_0) +\frac 12\xi_1 f_{c_j}(\xi_1)- \frac 12\xi_0 f_{c_j}(\xi_0)-\frac 12\int_{\xi_0}^{\xi_1} f_{c_j}(\xi)\ d\xi=0.
\end{equation*}
Then, considering the difference of these two equalities yields
$$
v_{c_1}^\prime(\xi_1) - v_{c_2}^\prime(\xi_1) - (v_{c_1}^\prime(\xi_0) - v_{c_2}^\prime(\xi_0)) - \frac 12\int_{\xi_0}^{\xi_1}( f_{c_1}(\xi)- f_{c_2}(\xi) )\ d\xi=0.
$$
This contradicts \eqref{non tangential 0}, \eqref{non tangential 1} and the situation of case (i). 

\vskip 1ex
\noindent
In case (ii), by integrating equation \eqref{ode selfsimilar} on $[\xi_0, \infty)$ for $f_{c_1}$ and $f_{c_2}$ and integrating by parts, we see that for $j =1, 2$
\begin{equation*}
\label{integrating by parts 3}
-v_{c_j}^\prime(\xi_0) - \frac 12\xi_0 f_{c_j}(\xi_0)-\frac 12\int_{\xi_0}^{\infty} f_{c_j}(\xi)\ d\xi=0.
\end{equation*}
Then, considering the difference of these two equalities yields
$$
- (v_{c_1}^\prime(\xi_0) - v_{c_2}^\prime(\xi_0)) - \frac 12\int_{\xi_0}^{\infty}( f_{c_1}(\xi)- f_{c_2}(\xi) )\ d\xi=0.
$$
This contradicts \eqref{non tangential 0} and the situation of case (ii). \qed

\vskip 2ex
\noindent
{\bf Proof of the existence and uniqueness of the solution of problem {\rm \eqref{ode selfsimilar on whole line}-\eqref{monotonicity on whole line}}. }
Let $c > 0$ and define $\psi : \mathbb R \to \mathbb R$ by
\begin{equation}
\label{definition of psi}
\psi(s) = \phi(c)-\phi(c-s)\ \mbox{ for } s \in \mathbb R.
\end{equation}
Then $\psi$ satisfies the same condition \eqref{nonlinearity} as $\phi$ does. It was shown in \cite{AtP} that, for every $a > 0$, there exists a unique $C^2$ solution $g_a = g_a(\xi)$ of the problem:
\begin{eqnarray}
&& \left( \psi^\prime(g_a) g_a^\prime\right)^\prime + \frac 12 \xi  g_a^\prime = 0\ \mbox{ in }\ [0,+\infty),
\\
&& g_a(0) = a,\quad g_a(\xi) \to 0\ \mbox{ as }\ \xi \to +\infty,
\\
&& g_a^\prime < 0\ \mbox{ in }\ [0,+\infty).\label{monotonicity psi}
\end{eqnarray}
Hence,  writing $V_a = V_a(\xi) = \psi\left(g_a(\xi)\right)$ for $\xi \in [0,+\infty)$ and proceeding similarly yield that 
\begin{eqnarray*}
&&-V_a^\prime(0) = \frac 12\int_0^\infty g_a(\xi)\ d\xi\ \mbox{ for } \ a > 0;\label{integral formula for derivative of V}
\\
&&0 <  g_{a_1} < g_{a_2}\ \mbox{  on } \ [0,+\infty) \ \mbox{ if }\ 0 < a_1 < a_2 < +\infty;\label{monotonicity of g with parameter}
\\
&& 0 > V_{a_1}^\prime(0) >  V_{a_2}^\prime(0) \ \mbox{ if }\ 0 < a_1 < a_2 < +\infty. \label{monotonicity of derivatives of V}
\end{eqnarray*}
For $a \in (0, c)$, define $f_{a,-} = f_{a,-}(\xi)$  by
\begin{equation*}
\label{on negative line}
f_{a,-}(\xi) = c - g_a(-\xi)\ \mbox{ for } \xi \in (-\infty, 0].
\end{equation*}
Then, in view of  \eqref{definition of psi}-\eqref{monotonicity psi}, $f_{a,-}$ satisfies the following:
\begin{eqnarray*}
&& \left( \phi^\prime(f_{a,-}) f_{a,-}^\prime\right)^\prime + \frac 12 \xi  f_{a,-}^\prime = 0\ \mbox{ in }\ (-\infty,0],\label{ode selfsimilar on negative line}
\\
&& f_{a,-}(0) = c-a,\quad f_{a,-}(\xi) \to c\ \mbox{ as }\ \xi \to -\infty,\label{boundary conditions on negative line}
\\
&& f_{a,-}^\prime < 0\ \mbox{ in }\ (-\infty,0].\label{monotonicity on negative line}
\end{eqnarray*}
Let  $f_{a,+} = f_{a,+}(\xi)\ (\xi \in [0,+\infty) )$  be the unique $C^2$ solution $f_{c-a}$ of problem \eqref{ode selfsimilar}-\eqref{monotonicity} where $c$ is replaced by $c-a$. Then we have
$$
\left(\phi(f_{a,-})\right)^\prime \left|_{\xi = 0}\right. = V_a^\prime(0)\ \mbox{ and }\ \left(\phi(f_{a,+})\right)^\prime \left|_{\xi = 0}\right. = v_{c-a}^\prime(0),
$$
where $v_{c-a}(\xi) = \phi(f_{a,+}(\xi))$ for $\xi \in [0,+\infty)$. Observe that both $V_a^\prime(0)$ and $v_{c-a}^\prime(0)$ are continuous as functions of $a$ on the interval $[0,c]$, $V_a^\prime(0)$ is strictly decreasing, $v_{c-a}^\prime(0)$ is strictly increasing, and $\lim\limits_{a \to 0}V_a^\prime(0) = \lim\limits_{a \to c}   v_{c-a}^\prime(0) = 0$.
Therefore, there exists a unique $a_* \in (0,c)$ satisfying $V_{a_*}^\prime(0) = v_{c-a_*}^\prime(0)$, and hence the unique $C^2$ solution $f_c= f_c(\xi)$ of problem {\rm \eqref{ode selfsimilar on whole line}-\eqref{monotonicity on whole line}} is given by
$$
f_c(\xi) = \left\{\begin{array}{rll}
  f_{a_*,+}(\xi)&\mbox{ if }\ \xi \in [0,+\infty),
\\
 f_{a_*,-}(\xi)&\mbox{ if }\ \xi \in (-\infty,0).
\end{array}\right. \qed
$$
 
\vskip 2ex
\noindent
{\bf Proof of  {\rm \eqref{integral formula for derivative of v cauchy}-\eqref{ monotonicity of derivatives of v cauchy}}}.  The proof of \eqref{integral formula for derivative of v} also works for \eqref{integral formula for derivative of v cauchy}. \eqref{integral formula for derivative of v cauchy} and \eqref{monotonicity of f with parameter cauchy} imply \eqref{ monotonicity of derivatives of v cauchy}. Thus it suffices to prove \eqref{monotonicity of f with parameter cauchy}.  Let $0 < c_1 < c_2 < +\infty$. Since
$\lim\limits_{\xi \to -\infty} f_{c_1}(\xi) = c_1 < c_2 = \lim\limits_{\xi \to -\infty} f_{c_2}(\xi)$, there exists $\xi_* < 0$ satisfying
$$
f_{c_1}(\xi) < f_{c_2}(\xi)\ \mbox{ for every } \xi \le \xi_*.
$$
Hence we can begin with supposing that there exists $\xi_0 > \xi_*$ satisfying
$$
f_{c_1}(\xi_0) = f_{c_2}(\xi_0)\ \mbox{ and }\ f_{c_1}(\xi) < f_{c_2}(\xi)\ \mbox{ for every } \xi \in [\xi_*, \xi_0).
$$
Therefore, the rest of the proof runs along that of \eqref{monotonicity of f with parameter}. \qed

\vskip 2ex
\noindent
{\bf Proof of  {\rm\eqref{from above and below on whole line}}}. In view of \eqref{monotonicity of f with parameter cauchy} and \eqref{ monotonicity of derivatives of v cauchy}, by continuity, we can find a sufficiently small $0 < \eta_\varepsilon << \varepsilon$ and two $C^2$ functions $f_\pm = f_\pm(\xi)$ for $\xi \in \mathbb R$ satisfying \eqref{scaling of f on whole line}, \eqref{monotonicity both on whole line}, \eqref{ modified ode on whole line} and the following:
\begin{eqnarray}
&& f_- < f_1 < f_+\ \mbox{ on } [0,+\infty);\label{inequality on positive line}
\\
&& 0 < f_{1-\frac 32\varepsilon} < \tilde{f}_- < f_{1-\frac 12\varepsilon} < f_{1+\frac 12\varepsilon}< \tilde{f}_+ <  f_{1+\frac 32\varepsilon}\ \mbox{ at } \xi = 0;\quad\label{value comparison at 0}
\\
&&0 > v^\prime_{1-\frac 32\varepsilon} > (\phi(\tilde{f}_-))^\prime > v^\prime_{1-\frac 12\varepsilon} > v^\prime_{1+\frac 12\varepsilon} > 
(\phi(\tilde{f}_+))^\prime  > v^\prime_{1+\frac 32\varepsilon}\ \mbox{ at } \xi =0,\quad\label{derivative comparison at 0}
\end{eqnarray}
where we put $\tilde{f}_\pm(\xi) = f_\pm(\xi\pm2\eta_\varepsilon^2)$ for $\xi \in \mathbb R$. Notice that $\tilde{f}_\pm = \tilde{f}_\pm(\xi)$ satisfy
\begin{equation}
\label{original ode}
 (\phi^\prime(\tilde{f}_\pm)\tilde{f}_\pm^\prime)^\prime + \frac 12\xi \tilde{f}_\pm^\prime = 0\ \mbox{ in } (-\infty,0].
 \end{equation}
 In order to prove \eqref{from above and below on whole line}, it suffices to show that
 \begin{equation}
 \label{4 inequalities}
 f_{1-\frac 32\varepsilon} < \tilde{f}_- < f_{1-\frac 12\varepsilon} \mbox{ and } f_{1+\frac 12\varepsilon}< \tilde{f}_+ <  f_{1+\frac 32\varepsilon} \ \mbox{ in } (-\infty,0].
 \end{equation}
Indeed,  \eqref{monotonicity both on whole line} implies that $f_-< \tilde{f}_-$ and $\tilde{f}_+ < f_+$ in $\mathbb R$, and hence \eqref{4 inequalities} and \eqref{monotonicity of f with parameter cauchy} give us
$$
f_-< f_1 < f_+\ \mbox{ on } (-\infty,0].
$$
Combining this with \eqref{inequality on positive line} yields that 
\begin{equation}
\label{inequality on whole line}
f_-< f_1 < f_+\ \mbox{ in }\mathbb R. 
\end{equation}
Also, since $\lim\limits_{\xi \to -\infty}f_\pm= \lim\limits_{\xi \to -\infty}\tilde{f}_\pm$, \eqref{4 inequalities} implies that 
\begin{equation}
\label{inequality at negative infinity}
1-\frac 32\varepsilon \le f_-(-\infty)\le 1-\frac 12\varepsilon < 1=f_1(-\infty) < 1+\frac 12\varepsilon \le f_+(-\infty) \le 1+\frac 32\varepsilon.
\end{equation}
 Therefore, \eqref{inequality on whole line} and \eqref{inequality at negative infinity} yield \eqref{from above and below on whole line}.
 
Thus, it remains to prove \eqref{4 inequalities}. \eqref{4 inequalities} consists of four inequalities. Since we will see that all the proofs are similar, let us prove the fourth one:
\begin{equation}
\label{4th inequality}
\tilde{f}_+ <  f_{1+\frac 32\varepsilon} \ \mbox{ in } (-\infty,0].
\end{equation}
By \eqref{value comparison at 0}, we have $\tilde{f}_+ <  f_{1+\frac 32\varepsilon}\ \mbox{ at } \xi = 0$. Hence, suppose that there exists $\xi_0 < 0$ satisfying
\begin{equation}
\label{on the interval}
\tilde{f}_+(\xi_0) = f_{1+\frac 32\varepsilon}(\xi_0)\ \mbox{ and } \tilde{f}_+ <  f_{1+\frac 32\varepsilon} \mbox{ on } (\xi_0, 0].
\end{equation}
Then,  by the uniqueness we also have 
\begin{equation}
\label{inequality of derivatives at the left-hand side}
v^\prime_{1+\frac 32\varepsilon} > (\phi(\tilde{f}_+))^\prime \ \mbox{ at } \xi = \xi_0.
\end{equation}
By \eqref{derivative comparison at 0}, we have 
\begin{equation}
\label{inequality of derivatives at the right-hand side}
(\phi(\tilde{f}_+))^\prime  > v^\prime_{1+\frac 32\varepsilon}\ \mbox{ at } \xi =0.
\end{equation}
Here,  integrating equations \eqref{original ode} for $\tilde{f}_+$ and \eqref{ode selfsimilar on whole line} for $f_{1+\frac 32\varepsilon}$ on the interval $[\xi_0,0]$, integrating by parts, considering the difference of the two resultant equalities, and using the fact that $\tilde{f}_+(\xi_0) = f_{1+\frac 32\varepsilon}(\xi_0)$, yield  that
\begin{equation*}
v^\prime_{1+\frac 32\varepsilon}(0) - (\phi(\tilde{f}_+))^\prime(0) - \left\{v^\prime_{1+\frac 32\varepsilon}(\xi_0) - (\phi(\tilde{f}_+))^\prime(\xi_0)\right\}
- \frac 12\int_{\xi_0}^0 (f_{1+\frac 32\varepsilon}(\xi) - \tilde{f}_+(\xi)) d\xi = 0.
\end{equation*}
On the other hand, by combining  \eqref{on the interval}, \eqref{inequality of derivatives at the left-hand side}, and \eqref{inequality of derivatives at the right-hand side}, we see that the left-hand side of this equality is negative, which is a contradiction.
Therefore, we get \eqref{4th inequality}. \qed
\par

In the next theorem, we prove 
a comparison principle over general domains including the case where their boundaries are unbounded,
by adjusting a proof that Bertsch, Kersner and Peletier gave for the Cauchy problem
(see \cite[Appendix, pp. 1005--1008]{BKP}). Observe that, when $\Om= \mathbb R^N$ 
(and \eqref{boundary inequality condition comparison} is dropped), there is no need to use
the approximating sequences $\{D_j\}$ and $\{ D_{j,k}\}$ constructed in our proof below,
since the sequence of balls $\{ B_{R_k}(0)\}$ suffices, as in \cite{BKP}.

\begin{theorem}{\rm (Comparison principle)}
\label{th:comparison principle}
Let  $T > 0$ and let $\Omega$ be a domain in $\mathbb R^N,$ with $N \ge 2$, where $\partial\Omega$ is not necessarily bounded.
Assume that $u, v \in C^{2,1}(\Omega\times (0,T]) \cap L^\infty(\Omega \times (0,T]) \cap C^0(\overline{\Omega}\times (0,T])$ satisfy the following:
\begin{eqnarray}
&&\partial_t u-\Delta \phi(u)\ \le \ \partial_t v-\Delta \phi(v)\ \ \mbox{ in } \Omega\times (0,T],\label{parabolic inequality comparison}
\\
&& u \le v \ \mbox{ on } \partial\Omega \times (0,T],\label{boundary inequality condition comparison}
\\
&& u(\cdot,t) \to u_0(\cdot)\ \mbox{ and } v(\cdot,t) \to v_0(\cdot)\ \mbox{ in } L^1_{loc}(\Omega)\ \mbox{ as } t \downarrow 0,\label{initial inequality condition comparison}
\end{eqnarray}
where $u_0,\ v_0 \in L^\infty(\Omega)$ satisfy the inequality $u_0 \le v_0$ in $\Omega$.

Then $u \le v$ in $\Omega\times (0,T]$.
\end{theorem}

\begin{proof} 
(a) {\bf Approximating the domain $\Om.$\ }
Let $d=d(x)$ be the distance of $x$ from the closed set $\mathbb R^N\setminus\Omega$ and let $U =\{ x \in \mathbb R^N : d(x) < 1 \}.$ From a lemma due to Calder\'on and Zygmund \cite[Lemma 3.6.1, p. 136]{Z} (see also \cite[Lemma 3.2, p. 185]{CZ}) it follows that there exist a function $\delta=\delta(x) \in C^\infty(U\cap\Omega)$ and a positive number $M = M(N)$ such that
\begin{equation}
\label{smooth approximation of distance}
M^{-1}d(x) \le \delta(x)\le Md(x) \ \mbox{ for all } x \in U\cap\Omega.
\end{equation}
\par
Since $\delta  \in C^\infty(U\cap\Omega)$, in view of \eqref{smooth approximation of distance} and the definition of $U$, Sard's theorem (see \cite{Sa, St}) yields that there exists a strictly decreasing sequence of positive numbers $\{\rho_j\}$ with $\lim\limits_{j\to\infty}\rho_j  = 0$ and $\rho_1 < M^{-1}$ such that every level set 
\begin{equation}
\label{smooth hypersurfaces}
\gamma_j = \{ x \in U\cap\Omega : \delta(x) = \rho_j \} 
\end{equation}
is a union of smooth hypersurfaces in $ \mathbb R^N$.
For each $j \in \mathbb N$, denote by $D_j$  the set satisfying $\partial D_j = \gamma_j$ and $\overline{D_j}\subset \Omega$ ($D_j$ is in general a union of smooth domains).  Moreover, in view of \eqref{smooth approximation of distance},  we may have
\begin{equation}
\label{increasing} 
 \overline{D_{j}} \subset D_{j+1} \ \mbox{ for }  j \in \mathbb N \ \mbox{ and } \ \Omega = \bigcup_{j=1}^\infty D_j.
\end{equation}
Without loss of generality, we may also assume that the origin belongs to all the $D_j$'s.
\par
The intersection $D_j \cap B_{R}(0)$ of $D_j$ with the ball $B_{R}(0)$ may not be
a finite union of Lipschitz domains; however, again by Sard's theorem, the restriction to
$\ga_j$ of the $C^\infty$-smooth map $x\mapsto |x|^2$ is regular at almost any of its values,
and hence there exists a strictly increasing and diverging sequence $\{ R_k\}$ of positive numbers
such that each $\pa B_{R_k}(0)$ is transversal to all the $\ga_j$'s;  thus, for each pair of $j$ and $k,$
$D_j \cap B_{R_k}(0)$ is a finite union of Lipschitz domains with piecewise $C^\infty$-smooth boundaries.
Therefore, by using a partition of unity, 
we can modify the boundary of $D_j \cap B_{R_k}(0)$ near the 
compact submanifold $\gamma_j\cap\partial B_{R_k}(0)$ to get a family 
$\{ D_{j,k}\} $ of finite unions of smooth domains, each one approximating $D_j \cap B_{R_k}(0),$ 
and satisfying the relations
\begin{eqnarray}
&& D_{j-1}\cap B_{R_{k}}(0) \subset D_{j,k} \subset  D_{j+1}\cap B_{R_{k}}(0)\ \left(\subset B_{R_k}(0) \right)\nonumber
\\
&&\mbox{ and }\ \partial D_{j,k} \cap \overline{D_{j-1}} =  \partial B_{R_k}(0) \cap \overline{D_{j-1}},\label{good approximation}
\end{eqnarray}
for every  $j \ge 2$ and $k \in \mathbb N.$
\par
(b) {\bf Constructing test functions.\ }
Set
$$
A = A(x,t) = \left\{\begin{array}{rl}
\frac {\phi(u)-\phi(v)}{u-v}  & \mbox{ if } u(x,t)\not= v(x,t), x \in \Omega, \mbox{ and } t > 0,
\\
\delta_1 &\mbox{ otherwise.}
\end{array}\right.
$$
Then $\delta_1 \le A \le \delta_2$ on $\mathbb R^{N+1}$ and we can 
approximate $A$ by a sequence $\{ A_n\}$ of regularizations satisfying
\begin{eqnarray}
&&A_n \in C^\infty(\mathbb R^{N+1})\ \mbox{ and } \delta_1 \le A_n \le \delta_2\ \mbox{ in } \mathbb R^{N+1} \mbox{ for each } n \in \mathbb N,\label{bounds parabolicity}
\\
&&A-A_n \to 0 \mbox{ in } L^2_{loc}( \mathbb R^{N+1}) \mbox{ as } n \to \infty.\label{L2convergence}
\end{eqnarray}

Let $0 < \tau < s < T$ and choose $\chi \in C_0^\infty(\mathbb R^N),$ 
with support $\supp\chi$ contained in $\Omega,$  such that $0 \le \chi \le 1$ in $\mathbb R^N.$   
In view of \eqref{smooth approximation of distance}, there exist $j_0, k_0 \in \mathbb N$ such that
\begin{equation}
\label{support of test function}
\supp\chi \subset D_{j,k} \ \mbox{ for every pair of } j \ge j_0 \mbox{ and } k \ge k_0.
\end{equation}
\par
Now, choose an integer $k \ge k_0$ and then a number $\varepsilon > 0$.  Since $u, v \in C^0(\overline{\Omega}\times (0,T])$, it follows from
\eqref{boundary inequality condition comparison}  that there exists $\mu > 0$ satisfying
\begin{equation}
\label{near the boundary in k+1}
\phi(u)\le \phi(v)+ \varepsilon\ \mbox{ in } \overline{\Omega_{\mu} \cap B_{R_{k+1}}(0)} \times [\tau,T],
\end{equation}
where  $\Omega_{\mu}$ is given by \eqref{neighborhood of boundary}. Hence, by \eqref{smooth approximation of distance} and  \eqref{good approximation}, we see  that there exists $j_1\ge j_0$ such that
\begin{equation}
\label{on the boundary close to partial Omega}
\phi(u)\le \phi(v)+ \varepsilon\ \mbox{ on } \left(\partial D_{j,k} \setminus D_{j-1}\right) \times [\tau,T]\ \mbox{ for every } j \ge j_1.
\end{equation}
\par
For each $j \ge j_1$ and $n\in\mathbb N$, let $w_{n,j} \in C^\infty(\overline{D_{j,k}}\times[0,s))  \cap C^0(\overline{D_{j,k}}\times[0,s])$ be the unique bounded solution of the problem:
\begin{eqnarray}
&& \partial_t w_{n,j} + A_n \Delta w_{n,j} = \delta_2 w_{n,j}\ \  \quad\mbox{ in } D_{j,k} \times [0,s),\label{adjoint equation}
\\
&&w_{n,j} = 0\ \ \quad\qquad\qquad\qquad\qquad\mbox{ on } \partial D_{j,k} \times [0,s), \label{homogeneous Dirichlet condition}
\\
&& w_{n,j}(x,s) = e^{-|x|} \chi(x)\qquad\qquad\mbox{ for every } x \in D_{j,k}.\label{terminal condition}
\end{eqnarray}
Then, by the parabolic regularity theory (see \cite{LSU}), we see that
$$
w_{n,j} \in  C^\infty\left((\overline{D_{j,k} }\times [0,s]) \setminus ( \{0\} \times \{s)\} )\right) \mbox{ and } \nabla w_{n,j} \in L^\infty(D_{j,k} \times [0,s]),
 $$
and, as in  \cite[Lemma B, p. 1007]{BKP}, we can prove the following lemma.
 \begin{lemma}
 \label{estimates for adjoint problems}
 There exists a  constant $c > 0$ depending only on $\chi$ such that, for each $j \ge j_1$ and $n\in\mathbb N$, the solutions $w_{n,j}$ have the following properties:
 \begin{eqnarray*}
 &&\mbox{\rm (i) }\ 0 \le w_{n,j} \le e^{-|x|}\ \mbox{ in } \overline{D_{j,k}}\times[0,s],
 \\
 &&\mbox{\rm (ii) }\ \int_0^sdt \int_{D_{j,k}} A_n (\Delta w_{n,j})^2 dx \le c,
 \\
 &&\mbox{\rm (iii) }\  \sup_{0 \le t \le s} \int_{D_{j,k}} |\nabla w_{n,j}(x,t)|^2 dx \le c,
 \\
 &&\mbox{\rm (iv) }\ 0 \le -\frac {\partial w_{n,j}}{\partial \nu} \le c e^{-R_k} \ \mbox{ on } \left(\partial D_{j,k} \cap \partial B_{R_k}(0) \right) \times [0,s], 
 \end{eqnarray*}
where $\nu$ denotes the unit outward normal vector to $\partial D_{j,k}$.
\end{lemma}
\begin{remark} {\rm The fact that $D_{j,k}\subset B_{R_k}(0)$ (see \eqref{good approximation}) guarantees that the same barrier function as in  \cite[Lemma B, p. 1007]{BKP} can be used to prove {\rm (iv)}. The proofs of the others are the same.}
\end{remark}
\par
(c) {\bf Completion of the proof.\ } For each $j \ge j_1$ and $n\in\mathbb N$,  multiplying \eqref{parabolic inequality comparison} by $w=w_{n,j}$ and integrating by parts the resultant inequality over $D_{j,k}\times [\tau, s]$ yield that 
\begin{eqnarray}
&&0 \ge \int\limits_{D_{j,k}\times[\tau,s]}\left\{ \pa_t(u-v)-\Delta\left[\phi(u)-\phi(v)\right]\right\} w\ dx\,dt \nonumber
\\
&&= \int\limits_{D_{j,k}}\left[(u-v)(x,t) w(x,t)\right]_\tau^s  dx 
- \int\limits_{D_{j,k}\times[\tau,s]} (u-v)\partial_t w\ dx\,dt\nonumber
\\
&&\qquad -\int\limits_\tau^s dt \int\limits_{\pa D_{j,k}} \frac {\partial}{\partial \nu}\left[\phi(u)-\phi(v)\right] w\ d\sigma + \!\!\!\int\limits_{D_{j,k}\times[\tau,s]} \nabla\left[\phi(u)-\phi(v)\right]\cdot \nabla w\, dx\,dt\nonumber
\\
&&= \int\limits_{D_{j,k}}\left\{(u-v)(x,s) e^{-|x|}\chi(x) - (u-v)(x,\tau) w(x,\tau) \right\}  dx\nonumber
\\
&&\qquad -\!\!\!\!\!\int\limits_{D_{j,k}\times[\tau,s]} (u-v)\pa_t w\,dx\,dt +\!\!\!\int\limits_{D_{j,k}\times[\tau,s]} \nabla\left[\phi(u)-\phi(v)-\veps\right]\cdot \nabla w\, dx\,dt;\label{1st integral inequality}
\end{eqnarray}
here we used \eqref{homogeneous Dirichlet condition} and \eqref{terminal condition}, and we modified the last term a little for later use. The last in \eqref{1st integral inequality} term equals
\begin{eqnarray*}
&& \int\limits_\tau^s dt \!\!\!\int\limits_{\partial D_{j,k} \setminus \overline{D_{j-1}} } \left[\phi(u)-\phi(v)-\varepsilon\right]\frac {\partial w}{\partial\nu}\,d\sigma
\\
&&\quad 
+ \int\limits_\tau^s dt \!\!\!\int\limits_{\partial D_{j,k} \cap \overline{D_{j-1}} } \left[\phi(u)-\phi(v)-\varepsilon\right]\frac {\partial w}{\partial\nu} \,d\sigma
- \!\!\!\int\limits_{D_{j,k}\times[\tau,s]} \left[\phi(u)-\phi(v)-\varepsilon \right]\,\Delta w\ dx.
\end{eqnarray*} 
Since $\frac {\partial  w}{\partial\nu}  \le 0$ on $\partial D_{j,k} \times [0,s],$  it follows from \eqref{on the boundary close to partial Omega} that the first term above is nonnegative; also, in the third term, we write:
$$
\phi(u)-\phi(v) -\varepsilon=\left\{A_n+(A-A_n)\right\}(u-v) -\varepsilon.
$$
Therefore, it follows from \eqref{1st integral inequality} and \eqref{adjoint equation} that
\begin{eqnarray}
&&0 \ge \int\limits_{D_{j,k}} (u-v)(x,s) e^{-|x|}\chi(x)\ dx - \int\limits_{D_{j,k}} (u-v)(x,\tau) w(x,\tau) \ dx\nonumber
\\
&&\quad + \int\limits_\tau^s dt\int\limits_{\partial D_{j,k} \cap \overline{D_{j-1}} } \left[\phi(u)-\phi(v)-\varepsilon\right]\frac {\partial w}{\partial\nu}\, d\sigma -\delta_2\!\!\!\int\limits_{D_{j,k}\times[\tau,s]} (u-v)w\, dx\,dt\nonumber
\\
&&\quad  - \int\limits_{D_{j,k}\times[\tau,s]} (u-v)(A-A_n)\Delta w\,dx\,dt 
+\varepsilon\!\!\! \int\limits_{D_{j,k}\times[\tau,s]} \Delta w\,dx\,dt.\label{2nd integral inequality}
\end{eqnarray}

Since $u$ and $v$ are bounded, there exists a constant $K > 0$ such that
$$
\max\{|u-v|, |\phi(u)-\phi(v) -\varepsilon| \} \le K\ \mbox{ in } \Omega \times [0,T].
$$
Combining \eqref{good approximation} with (iv) of Lemma \ref{estimates for adjoint problems} yields that  the third term in \eqref{2nd integral inequality} is bounded from below by
$$
-cKe^{-R_k}TN\omega_N R_k^{N-1},
 $$
 where $\omega_N$ is the volume of the unit ball in $\mathbb R^N$. By using (i) of Lemma \ref{estimates for adjoint problems}, we see that the fourth term in \eqref{2nd integral inequality} is bounded from below  by
 $$
 -\delta_2\int\limits_\tau^s dt \int\limits_\Omega \max\{u-v,0\}e^{-|x|}\ dx.
 $$
 With the aid of \eqref{bounds parabolicity}, (ii) of Lemma  \ref{estimates for adjoint problems} yields that the fifth and the sixth terms   in \eqref{2nd integral inequality} are bounded from below  by
$$
-K \frac {\sqrt{c}}{\sqrt{\delta_1}}\left(\int\limits_0^T dt\int\limits_{D_{j,k}} (A-A_n)^2\ dx\right)^{\frac 12}\ \mbox{ and } -\varepsilon\frac {\sqrt{c}}{\sqrt{\delta_1}}\sqrt{T|D_{j,k}|},
$$
respectively, where $|D_{j,k}|$ denotes the $N$-dimensional Lebesgue measure of $D_{j,k}$. Consequently, 
with these bounds and by using \eqref{support of test function} 
in the first term in \eqref{2nd integral inequality}, from \eqref{2nd integral inequality} we obtain:
\begin{eqnarray*}
&&  \int\limits_{\Omega} (u-v)(x,s) e^{-|x|}\chi(x)\ dx \le \int\limits_{D_{j,k}} (u-v)(x,\tau) w_{n,j}(x,\tau) \ dx\\
&& \qquad + cKe^{-R_k}TN\omega_N R_k^{N-1} + \delta_2\int\limits_\tau^s dt \int\limits_\Omega \max\{u-v,0\}e^{-|x|}\ dx\\
&&\qquad +K \frac {\sqrt{c}}{\sqrt{\delta_1}}\left(\int\limits_0^T dt\int\limits_{D_{j,k}} (A-A_n)^2\ dx\right)^{\frac 12} + \varepsilon\frac {\sqrt{c}}{\sqrt{\delta_1}}\sqrt{T|D_{j,k}|}.
 \end{eqnarray*}
\par
Since $\varepsilon > 0$ is arbitrarily chosen and $D_{j,k} \subset B_{R_{k}}(0)$, we can remove the last term in the above inequality. Also, letting $n \to \infty$ and $\tau \to 0$ with in mind \eqref{L2convergence} and 
\eqref{initial inequality condition comparison}, respectively, yield that
\begin{equation*}
\int\limits_{\Omega} (u-v)(x,s) e^{-|x|}\chi(x)\ dx \le  cKe^{-R_k}TN\omega_N R_k^{N-1} + \delta_2\!\!\!\int\limits_{\Om\times[0,s]} \max\{u-v,0\}e^{-|x|}\,dx\,dt.
\end{equation*}
By letting $k \to \infty$, we remove the first term in the right-hand side of this inequality. Then, since $\chi \in C^\infty_0(\mathbb R^N)$ is an arbitrary function satisfying that $0 \le \chi \le 1$ in $\mathbb R^N$ and its support is contained in $\Omega$,  we conclude that for every $s \in [0,T]$
\begin{equation}
\label{final integral inequality}
 \int\limits_{\Omega} \max\{(u-v)(x,s), 0\} e^{-|x|}\ dx \le \delta_2\int\limits_0^s dt \int\limits_\Omega \max\{u-v,0\}e^{-|x|}\ dx.
 \end{equation}
\par
Finally, Gronwall's lemma implies that $u \le v$ in $\Omega\times (0,T]$.
\end{proof}

\vskip 4ex
\bigskip
\noindent{\large\bf Acknowledgement.}
\smallskip

The authors would like to thank Professor Hitoshi Ishii for
the idea introducing the sequence of balls $\{ B_{R_k}(x_k) \}_{k=1}^\infty$ in the end of Section \ref{section3} and the use of sup- and infconvolutions in \eqref{supconvolution} and \eqref{infconvolution} in Section \ref{section simple application}.

\end{document}